\documentclass[10pt,a4paper]{article}

\usepackage{url}
\usepackage{verbatim}
\usepackage{latexsym}
\usepackage{amssymb,amstext,amsmath,amsthm}
\usepackage{epsf}
\usepackage{epsfig}
\usepackage{a4wide}
\usepackage{verbatim}
\usepackage{proof}
\usepackage{latexsym}
\newtheorem{theorem}{Theorem}[section]
\newtheorem{corollary}{Corollary}[theorem]
\newtheorem{lemma}{Lemma}[theorem]

\usepackage{float}
\floatstyle{boxed}
\restylefloat{figure}

\def\oge{\leavevmode\raise
.3ex\hbox{$\scriptscriptstyle\langle\!\langle\,$}}
\def\feg{\leavevmode\raise
.3ex\hbox{$\scriptscriptstyle\,\rangle\!\rangle$}}

\newcommand{\pair}[1]{{\langle #1 \rangle}}

\def\N0{\hbox{\sf N}_0}

\newcommand{\ort}[1]{{#1}^{\perp}}

\newcommand{\del}[1]{}

\begin{document}

\title{The Regular Element Property in Constructive Mathematics}

\author{Thierry Coquand}
\date{Computer Science Department, University of Gothenburg}
\maketitle


\section*{Introduction}

The goal of this note is to present Kaplansky's proof of the Regular Element Property \cite{Kap} and to explain
how this argument can be adapted to the case of a coherent, strongly discrete and Noetherian (with an inductive
definition of Noetherian \cite{JL}) rings in a constructive setting. We thus get, in this setting, an algorithm
which given a f.g. regular ideal $I = \pair{a_1,\dots,a_n}$, build a regular element in $I$.
This complements in a way the previous treatment by Richman \cite{Richman}.

One maybe interesting feature of the constructive reformulation is that it uses the technique of point-free open
induction \cite{open,berger}. (A previous inductive proof of the nilregular element property \cite{nil} used a
simpler form of induction.)

\section{Classical Proof}

 We recall that {\em incomparable} prime ideals $P_1,\dots,P_m$ are {\em independent}: we cannot have $\cap_{i\neq k}P_i\subseteq P_k$
 since this would imply that we have  $i\neq k$ such that $P_i\subseteq P_k$. 

Let $\ort{v}$ be the ideal of elements $x$ such that $vx = 0$. An ideal of the from $\ort{v}$ for $v\neq 0$
is called {\em orthogonal} ideal.

 Each {\em maximal} orthogonal ideal $\ort{v}$ is a {\em prime} ideal: if $ab$ in $\ort{v}$ then
 $abv=0$ and $av\neq 0$ then $\ort{(av)} = \ort{v}$ by maximality  and hence $b$ in $\ort{v}$.

 So if we have $v_1,\dots,v_m$ non zero elements such that each
 $\ort{v_i}$ is orthogonal {\em and} maximal, they are necessarily incomparable and prime
 and so the ideals $\ort{v_i}$ are independent: $\cap_{i\neq k}\ort{v_i}$ is not
 a subset of $\ort{v_k}$.

 Using Noetherianity, we see that the set of maximal orthogonal ideals has to be finite: if $v_{m+1}$ is in
 $\pair{v_1,\dots,v_m}$ then $\cap_{i\leqslant m} \ort{v_i}$ is a subset of $\ort{v_{m+1}}$
 and then $v_1,\dots,v_m,v_{m+1}$ cannot be independent.

 We consider then the finite set $\ort{v_1},\dots,\ort{v_m}$ of {\em all} maximal orthogonal ideals. It is an independent set.
 So, for each $k$, we have $u_k$ such that $u_kv_k \neq 0$ and $u_kv_i = 0$ if $i\neq k$.

 \medskip

 Let $I = \pair{a_1,\dots,a_n}$ be a regular ideal. For each $k$ we have $i_k$ such that $u_kv_ka_{i_k}\neq 0$ since $u_kv_k\neq 0$.
 Let $a = \Sigma_k u_ka_{i_k}$ element of $I$. We have $av_k = u_kv_ka_{i_k}\neq 0$. I claim that $a$ is regular.
 Indeed if $va = 0$ and $v\neq 0$ we have $a$ in $\ort{v}$ and $a$ is not in any $\ort{v_1},\dots,\ort{v_m}$ which
 contradicts the fact that this is a list of all maximal orthogonal ideals.
 
\section{Constructive Version}

We are going to see that the classical argument of the first section can be read as an algorithm. The idea is to proceed
{\em dynamically}: we use a finite set $\ort{v_1},\dots,\ort{v_m}$ which is an approximation of the set of all maximal orthogonal
ideals to compute a candidate for a regular element in a given regular ideal. If it is not regular, we use this to compute
a better approximation. The problem is then reduced to showing that this process will end eventually.

\medskip

If $S$ is a relation on a set $X$, we say that $x$ is $S$-accessible if, and only if,  all $y$ such that $S(x,y)$ are
$S$-accessible. The relation $S$ is well-founded if, and only if,  all elements are accessible.

\medskip

We assume now that the ring $R$ is {\em coherent} and {\em strongly discrete}. Each ideal $\ort{v}$ is then a f.g. ideal.

We can then decide inclusion between f.g. ideals. We write $I<J$ for the fact that $I$ is a {\em strict} subset of $J$.

We assume that $<$ is well-founded.

We consider then following relation: $v_1,\dots,v_n \prec w_1,\dots,w_m$
(for $m\leqslant n$) if, and only if,  $v_i=w_i$ for $i<m$ and $\ort{v_m}<\ort{w_m}$.

A list $v_1,\dots,v_n$ is called independent if the corresponding ideals $\ort{v_i}$ are independent.
It is non decreasing if $\ort{v_i}\subseteq \ort{v_j}$ implies $i\leqslant j$. Clearly any independent list is
non decreasing. 

If $\alpha, v$ is independent then so is $\alpha$. So the independent lists form a tree, and the branching of
the tree is given by a decidable subset of the set $R$.


\begin{lemma}\label{wf1}
  The relation $\prec$ is well-founded.
\end{lemma}

\begin{proof}
  We show by induction on $v$ and the length of $\alpha$ that any element $v\alpha$ if $\prec$-accessible.
  Indeed,
  we have $v\alpha \prec \beta$ if, and only if,  $\beta = w$ for some $w>v$ or $\beta = v\gamma$ with $\alpha\prec \gamma$.
\end{proof}

\begin{corollary}
 If $\alpha$ is a non decreasing list, there exists an independent $\beta$ such that $\alpha\prec^* \beta$.
\end{corollary}

\begin{proof}
  Indeed if $\alpha$ non decreasing is not independent, we find $\beta$ non decreasing such
  that $\alpha \prec \beta$. We conclude by Lemma \ref{wf1}.
\end{proof}

We now consider the lexicographic
ordering $\alpha<\beta$ if, and only if, $\beta$ is an extension of some $\gamma$ such that $\alpha\prec\gamma$
or a strict extension of $\alpha$. 

\begin{lemma}\label{wf2}
  The ordering $<$ is well-founded on {\em independent} lists.
\end{lemma}

\begin{proof}
  We use the point-free technique of open induction \cite{open}.

  If $\alpha$ is a list, let $\pair{\alpha}$ the f.g. ideal it generates.

  Let $M$ be the set of independent $\alpha$ such that
  $\alpha\prec \beta$ implies that $\beta$ is $<$-accessible.
  If $\alpha\rightarrow \alpha v$ and $\alpha v$ is in $M$ then so is $\alpha$ since
  $\alpha \prec \beta$ implies $\alpha v \prec \beta$. So $M$ defines a subtree in the tree of
  independent sequences.

  \medskip

  This tree has a complex branching structure, since the branching is now given by the 
  predicate of being $\prec$-accessible\footnote{This is what is interesting from a constructive point of view,
  the argument is a priori in ID$_2$ and not ID$_1$.}.

  \medskip  
  
  The relation of adding one element
  $\alpha\rightarrow \alpha v$ is well-founded on the set $M$, since
  we then have $\pair{\alpha}<\pair{\alpha v}$ and the ring $R$ is Noetherian.

  This means that the following induction principle holds: is $\psi(\alpha)$ is such that
  $\psi(\alpha)$ holds whenever $\psi(\alpha v)$ holds for all $\alpha v$ in $M$, then
  $\psi$ holds for the empty sequence.

  We apply this for the case where $\psi(\alpha)$ expresses that $\alpha$ is accessible
  for the lexicographic ordering.

  \medskip

  We take $\alpha$ in $M$ and assume that $\alpha v$ in $M$ implies $\alpha v$ accessible.

  Then we show that $\alpha$ is accessible. For this we show that $\alpha v$ is accessible whenever
  $\alpha v$ is independent by induction on $v$.

  Let $v$ be such that $\alpha w$ is accessible for each $w>v$. We show that $\alpha v$ is accessible.

  It is enough to show that $\alpha v$ is in $M$.
  Assume $\alpha v\prec \beta$ then either $\alpha\prec \beta$ (and then $\beta$ is accessible since
  $\alpha$ is in $M$) or $\beta = \alpha w$ with $v<w$, and it is accessible by hypothesis.

  By the induction principle on $M$, it follows that the empty sequence is accessible.
\end{proof}

 We can also extract from the classical section the following algorithm, which holds for any ring with a decidable equality.
Given $I = \pair{a_1,\dots,a_m}$ which is regular.

\begin{lemma}\label{ind}
  If $\alpha = v_1,\dots,v_n$ if an independent sequence, then we can build $a\neq 0$ in $I$ such that
  $av_i\neq 0$ for each $i$.
\end{lemma}

\medskip

We deduce from all this the following algorithm for computing a regular element in a coherent, strongly discrete
and inductive Noetherian rings.

We start from the empty sequence $\alpha_0 = ()$. Since $I$ is regular we can find $u_0$ in $I$ which is $\neq 0$.
Then either
$u_0$ is regular or we find $\alpha_1>\alpha_0$ such that $\alpha_1$ is independent. Using Lemma \ref{ind}, we compute
$u_1\neq 0$ in $I$ such that $u_1v\neq 0$ for each $v$ in $\alpha_1$. Either $u_1$ is regular, or we find $v\neq 0$
such that $u_1v = 0$ and $\alpha_2>\alpha_1 v$
which is independent, using Corollary of Lemma \ref{wf1}, and so on.

 This process should eventually produces a regular element  in $I$ using Lemma \ref{wf2}.

\end{document}